\documentclass[12pt,twoside]{amsart}
\usepackage{mathrsfs}
\usepackage{enumitem}
\usepackage{amsmath, amsfonts, amssymb, amsthm, amscd}
\usepackage[all]{xy}
\usepackage{fancyhdr}
\usepackage{hyperref}
\usepackage{geometry}

\usepackage[mathscr]{eucal}

\theoremstyle{plain}
\newtheorem{thm}{Theorem}[section]
\newtheorem{inthm}{Theorem}

\newtheorem{theorem}[thm]{Theorem}
\newtheorem{lemma}[thm]{Lemma}
\newtheorem{corollary}[thm]{Corollary}
\newtheorem{proposition}[thm]{Proposition}

\theoremstyle{remark}

\newtheorem{defn-thm}[thm]{Definition-Theorem}
\newtheorem{defn-lem}[thm]{Definition-Lemma}

\renewcommand{\bar}{\overline}

\renewcommand{\phi}{\varphi}

\newcommand{\C}{{\mathbb C}}

\newcommand{\R}{{\mathbb R}}

\newcommand{\Q}{{\mathbb Q}}

\newcommand{\M}{{\mathcal M}}
\newcommand{\T}{{\mathcal T}}

\renewcommand{\tilde}{\widetilde}
\newcommand{\p}{{\Phi}}

\begin{document}

\def\dW{\mbox{diff\:}\times \mbox{Weyl\:}}
\def\End{\operatorname{End}}
\def\Hom{\operatorname{Hom}}
\def\Aut{\operatorname{Aut}}
\def\Diff{\operatorname{Diff}}
\def\im{\operatorname{im}}
\def\tr{\operatorname{tr}}
\def\Pr{\operatorname{Pr}}
\def\Z{\bf Z}
\def\O{\mathcal{O}}
\def\CP{\mathbb{C}\mathbb{P}}
\def\P{\Phi}
\def\TT{\mathcal {T}_m^H}

\def\Q{\bf Q}
\def\R{\bf R}
\def\C{\mathbb{C}}
\def\H{H_{\mathrm{pr}}}
\def\Hil{\mathcal{H}}
\def\proj{\operatorname{proj}}
\def\id{\mbox{id\:}}
\def\a{\mathfrak a}
\def\d{\partial}
\def\tO{\tilde{\Omega}}

\def\b{\beta}
\def\c{\gamma}
\def\p{\partial}
\def\f{\frac}
\def\i{\sqrt{-1}}
\def\t{\tau}
\def\T{\mathcal{T}}
\def\Tan{\mathrm{T}^{1,0}}
\def\aTan{\mathrm{T}^{0,1}}
\def\Kahler{K\"{a}hler\:}
\def\w{\omega}
\def\X{\mathcal{X}}
\def\Y{\mathcal{Y}}
\def\K{\mathcal {K}}
\def\m{\mu}
\def\M{\mathcal {M}}
\def\Z{\mathcal {Z}_m}
\def\ZZ{\mathcal {Z}_m^H}

\newcommand{\bp}{\bar{\partial}}

\def\v{\nu}
\def\D{\mathcal{D}}
\def\U{\mathcal {U}}
\def\V{\mathcal {V}}
\def\Omegak{\frac{1}{k!}\bigwedge\limits^k\mu\lrcorner\Omega}
\def\Omegakp{\frac{1}{(k+1)!}\bigwedge\limits^{k+1}\mu\lrcorner\Omega}
\def\Omegakpp{\frac{1}{(k+2)!}\bigwedge\limits^{k+2}\mu\lrcorner\Omega}
\def\Omegakm{\frac{1}{(k-1)!}\bigwedge\limits^{k-1}\mu\lrcorner\Omega}
\def\Omegakmm{\frac{1}{(k-2)!}\bigwedge\limits^{k-2}\mu\lrcorner\Omega}
\def\Omegakk{\Omega_{i_1,i_2,\cdots,i_k}}
\def\Omegakkp{\Omega_{i_1,i_2,\cdots,i_{k+1}}}
\def\Omegakkpp{\Omega_{i_1,i_2,\cdots,i_{k+2}}}
\def\Omegakkm{\Omega_{i_1,i_2,\cdots,i_{k-1}}}
\def\Omegakkmm{\Omega_{i_1,i_2,\cdots,i_{k-2}}}
\def\mukm{\frac{1}{(k-1)!}\bigwedge\limits^{k-1}\mu}
\def\sumk{\sum\limits_{i_1<i_2<,\cdots,<i_k}}
\def\sumkm{\sum\limits_{i_1<i_2<,\cdots,<i_{k-1}}}
\def\sumkmm{\sum\limits_{i_1<i_2<,\cdots,<i_{k-2}}}
\def\sumkp{\sum\limits_{i_1<i_2<,\cdots,<i_{k+1}}}
\def\sumkpp{\sum\limits_{i_1<i_2<,\cdots,<i_{k+2}}}
\def\Omegakb{\Omega_{i_1,\cdots,\bar{i}_t,\cdots,i_k}}
\def\Omegakmb{\Omega_{i_1,\cdots,\bar{i}_t,\cdots,i_{k-1}}}
\def\Omegakpb{\Omega_{i_1,\cdots,\bar{i}_t,\cdots,i_{k+1}}}
\def\Omegakt{\Omega_{i_1,\cdots,\tilde{i}_t,\cdots,i_k}}

\title{Limits of Hyperk\"ahler Manifolds and Moduli Spaces of Sheaves on K3 Surfaces}
\author{Kefeng Liu}
\address{Mathematical Sciences Research Center, Chongqing University of Technology, Chongqing 400054, China; \newline
Department of Mathematics,University of California at Los Angeles, Los Angeles, CA 90095-1555, USA}
\email{liu@math.ucla.edu}

\author{Yang Shen}
\address{Mathematical Sciences Research Center, Chongqing University of Technology, Chongqing 400054, China}
\email{syliuguang2007@163.com}
\date{}

\begin{abstract}
We prove that a compact strongly Gauduchon manifold which is a limit of
compact hyperk\"ahler manifolds is bimeromorphic to a compact
hyperk\"ahler manifold. This settles a case of a conjecture of
Soldatenkov and Verbitsky. 
A closed completion of a Gauduchon metric bounds the Dirichlet
energies of nearby biholomorphisms. This bound yields a meromorphic
limit map, which has degree one and is therefore a bimeromorphism
onto a compact hyperk\"ahler manifold.
We apply this
result to moduli spaces of stable torsion-free coherent sheaves on K3
surfaces. A Bott--Chern diagonal formula gives the
$\partial\bar\partial$-lemma, and positivity of determinant
classes gives K\"ahlerness, proving Perego's conjecture.
\end{abstract}
\maketitle
\parskip=5pt
\baselineskip=15pt
\tableofcontents

\setcounter{section}{-1}
\section{Introduction}
Moduli spaces of stable sheaves on K3 surfaces are a principal
source of compact hyperk\"ahler manifolds. By a compact
hyperk\"ahler manifold we mean a simply connected compact
K\"ahler manifold whose space of holomorphic two-forms is
generated by an everywhere nondegenerate form. The few known
deformation types make such constructions important. Even
within a fixed deformation type, a geometric description of a
general member is difficult: for $n\ge2$, Hilbert schemes
$S^{[n]}$ obtained by varying the K3 surface form a
$20$-dimensional family, whereas $S^{[n]}$ has a
$21$-dimensional local deformation space
\cite[Section~2.4]{Hu1}.

For an arbitrary K3 surface, Perego--Toma \cite{PT} constructed
compact smooth moduli spaces with holomorphic symplectic forms
and determined their deformation types. Their K\"ahlerness
remained open beyond the projective case and several special
cases. We prove Perego's conjecture
\cite[Conjecture~1.1]{Per} that these moduli spaces are
K\"ahler. The proof resolves a gap in Perego's application of
Bishop's theorem by replacing the required graph-volume bound
with a Dirichlet energy bound.

Write $M_v(S,\omega)$ for the moduli space of
$\mu_\omega$-stable torsion-free coherent sheaves on a K3
surface $S$ with Mukai vector $v$.

\begin{inthm}\label{main moduli Kahler}
Let $S$ be a compact K3 surface, let $v=(r,\xi,a)$ be a Mukai vector
with $r>0$, $\xi\in\operatorname{NS}(S)$ and $r,\xi$ relatively prime,
and let $\omega$ be a $v$-generic K\"ahler class. If $M_v(S,\omega)$
is nonempty and has positive dimension, it is a compact hyperk\"ahler
manifold of $K3^{[n]}$-type, where $2n=v^2+2$.
\end{inthm}

The projective case, the case $v^2=0$, and the case in which all
parametrized sheaves are locally free were known; see
\cite{PT,Per}. In the last case, a K\"ahler metric comes from
the theory of Hermite--Einstein connections. The theorem
establishes K\"ahlerness for the full compact moduli space,
including its non-locally-free locus. The deformation type
was already known; the conclusion completes this construction
over nonprojective K3 surfaces.

The proof rests on a general result about smooth limits
of compact hyperk\"ahler manifolds. K\"ahlerness is stable
under small deformations \cite{Kodaira05}, but need not persist at
the central fiber of a smooth family: non-K\"ahler Mukai
transforms occur as limits of compact hyperk\"ahler manifolds
\cite[Example~21.9]{GHJ}. One may therefore ask whether the
limit belongs to the Fujiki class $\mathcal C$, the class of
compact complex manifolds bimeromorphic to compact K\"ahler
manifolds.

Following \cite[Definition~1.12(iii)]{Per}, we call a compact
complex manifold $X$ a limit of compact hyperk\"ahler manifolds
if there is an analytic family of compact complex manifolds
$\pi:\mathcal X\to B$ over a smooth connected complex base
$B$, with $X\simeq X_0$ and compact hyperk\"ahler fibers
$X_i=X_{t^i}$ for some sequence $t^i\to0$ in $B$.

A compact complex manifold $X$ of dimension $d\ge2$ is
\emph{strongly Gauduchon} if it admits a real closed $(2d-2)$-form
whose $(d-1,d-1)$ component is strictly positive
\cite[Lemma~1.1]{PopoviciSG}. Equivalently, it admits a Hermitian
metric $\gamma$ for which $\partial\gamma^{d-1}$ is
$\bar\partial$-exact; see Lemma~\ref{closed Gauduchon completion}.

\begin{inthm}\label{main ddbar HK limit}\label{ddbar HK limit}
Let $X$ be a compact strongly Gauduchon manifold which is a limit of
compact hyperk\"ahler manifolds. Then $X$ is bimeromorphic to a
compact hyperk\"ahler manifold.
In particular, $X$ belongs to the Fujiki class $\mathcal C$.
\end{inthm}

The strongly Gauduchon hypothesis holds if $X$ satisfies the
$\partial\bar\partial$-lemma or $H^1(X,\mathcal O_X)=0$;
see Lemma~\ref{closed Gauduchon completion}.

For an analytic family over a disk centered at zero whose
noncentral fibers are compact hyperk\"ahler manifolds, this
settles a case of the conjecture of Soldatenkov--Verbitsky
\cite[Conjecture~1.2]{SV24}. Their conjecture asks whether a
holomorphically symplectic central fiber belongs to the Fujiki
class $\mathcal C$. When the central fiber is strongly Gauduchon,
our theorem gives a compact hyperk\"ahler
bimeromorphic model; holomorphic symplecticity of the central
fiber is established in the proof. Soldatenkov--Verbitsky prove
K\"ahlerness when the central fiber is Hermitian symplectic
\cite[Theorem~1.4]{SV24}.

Period maps and Verbitsky's global Torelli theorem
\cite{Verbitsky13,VerbitskyErratum} relate the deformation and
birational geometry of compact hyperk\"ahler manifolds; see also
\cite{Hu1,HuybrechtsTorelli,Markman11}. Using Perego's period
construction \cite[Proposition~2.4]{Per}, we obtain sequences
$$
X_i=X_{t^i},\qquad Y_i=Y_{{t'}^i},\qquad
 t^i\to0,\quad {t'}^i\to0,\qquad
 f_i:X_i\xrightarrow{\sim}Y_i,
$$
where $Y=Y_0$ is a compact hyperk\"ahler manifold. Choose
K\"ahler forms $\omega_i\to\omega$ on $Y_i\to Y$ in
$C^\infty$. The marked classes $[f_i^*\omega_i]$ are bounded.

Perego's proof for moduli spaces that are $b_2$-manifolds
\cite[Definition~1.4, Theorem~1.13]{Per} uses Bishop's theorem
to pass to a limit of the graphs $\Gamma_i$ of the $f_i$.
This requires a volume bound in a fixed Hermitian metric.
The bound in \cite[Lemma~2.5]{Per} uses varying K\"ahler
forms whose cohomology classes converge; the limiting class
on the source need not be K\"ahler, so this does not give the required
bound. As observed by Soldatenkov--Verbitsky
\cite[Section~1.2]{SV24}, this is the gap in \cite{Per}.
In a fixed metric, the graph-volume expansion contains
intermediate powers of $f_i^*\omega_i$ whose integrals are
not cohomological and are not controlled by the convergence
of these classes.

The strongly Gauduchon hypothesis in Theorem~\ref{main ddbar HK limit}
gives a bound for the Dirichlet energies. Put $d=\dim_{\C}X_0$.
Choose a Hermitian metric $\gamma$ on $X_0$ and a real $(2d-2)$-form
$\Psi$ satisfying
$$
d\Psi=0,\qquad \Psi^{d-1,d-1}=\gamma^{d-1}.
$$
For Hermitian forms $\gamma_i\to\gamma$ in $C^\infty$, it
follows that
$$
\int_{X_i}f_i^*\omega_i\wedge\gamma_i^{d-1}
\le C\langle[f_i^*\omega_i],[\Psi]\rangle\le C'.
$$
We prove a compactness theorem for holomorphic maps of
bounded Dirichlet energy between smoothly varying compact
sources and K\"ahler targets
(Proposition~\ref{finite energy compactness}).
Local regularity and meromorphic extension give a limiting
meromorphic map; the symplectic forms prevent a drop in rank,
and injectivity of the $f_i$ gives degree one. This yields
the required bimeromorphism.

To apply the limit theorem to $M=M_v(S,\omega)$, we first
prove a Bott--Chern refinement of Markman's diagonal formula
\cite{MarkmanGenerators}:
$$
[\Delta_M]_{BC}=\sum_j a_j\boxtimes b_j.
$$
Here $a_j,b_j$ are Bott--Chern classes of pure bidegree.
The identity implies that $H^{p,q}_{BC}(M)\to H^{p,q}_A(M)$
is injective in every bidegree, hence that $M$ satisfies the
$\partial\bar\partial$-lemma. Thus $M$ is strongly Gauduchon,
and its Hodge decomposition is obtained without assuming K\"ahlerness
\cite{DGMS,AT13}. The relative moduli construction of
\cite{PT,Per} realizes $M$ as a limit of compact
hyperk\"ahler manifolds, so Theorem~\ref{ddbar HK limit}
gives a bimeromorphic compact hyperk\"ahler model.

It remains to prove K\"ahlerness of $M$ itself. For the
Donaldson class $A$ and a second determinant class $B$, we
prove
$$
A\cdot C\ge0,\qquad
A\cdot C=0\ \Longrightarrow\ B\cdot C>0
$$
for every compact curve $C\subset M$. Together with the
positive-cone and genericity conditions verified in
Section~\ref{moduli proof section}, these inequalities give
K\"ahlerness by Proposition~\ref{two classes Kahler criterion}.
The curvewise argument includes families of non-locally-free
sheaves and avoids deforming a given rational curve along
the Hodge locus of its class.

The paper follows this two-stage argument. After the
preliminaries in Section~\ref{com section}, we first prove the
limit theorem. Section~\ref{ddbar HK limits section} establishes
holomorphic symplecticity of the limit and constructs the
comparison biholomorphisms by period theory.
Section~\ref{holomorphic maps section} develops the energy
estimates and meromorphic compactness, which are combined with
the degree-one argument in Section~\ref{HK limit proof section}
to prove Theorem~\ref{main ddbar HK limit}. We then turn to the
moduli spaces: Section~\ref{moduli Kahler section} proves the
Bott--Chern diagonal formula and the $\partial\bar\partial$-lemma,
and Section~\ref{Donaldson positivity section} establishes the
positivity of the determinant classes on curves. Finally,
Section~\ref{moduli proof section} applies the K\"ahlerness
criterion of Section~\ref{Kahler criterion section} to these
classes, completing the proof of Theorem~\ref{main moduli Kahler}.

\section{Preliminaries}\label{com section}
\subsection{Deformations of compact complex manifolds}
An analytic family of compact complex manifolds over a complex
manifold $B$ is a proper holomorphic submersion $\pi:\X\to B$.
A deformation of a compact complex manifold $X$ is such a family
together with a point $0\in B$ and an identification $X_0\simeq X$.
We write $X=X_0$ and assume $\dim_{\C}X=d$. After shrinking
$B$, fix differentiable
trivializations $d_t:X_t\to X$. We use the induced Gauss--Manin
identifications and write
$$
t^i\longrightarrow0,\qquad X_i=X_{t^i}.
$$
Fix a finite coordinate covering
$$
\mathscr U=\{(V_\alpha\Subset U_\alpha;z_\alpha)\},
\qquad \bigcup_\alpha V_\alpha=X.
$$
Under the differentiable identifications, the same open sets
cover $X_t$. Choose fiberwise holomorphic coordinates
$\zeta_{t,\alpha}$ with $\zeta_{0,\alpha}=z_\alpha$ and
$\zeta_{t,\alpha}\to z_\alpha$ in $C^\infty$ as $t\to0$.
The corresponding Beltrami differentials satisfy
$$
\begin{gathered}
\phi(t)\in A^{0,1}(X,\Tan X),\qquad
\phi(t)\longrightarrow0\text{ in }C^\infty,\\
\mathrm T^{0,1}X_t
=\{\bar v-\phi(t)(\bar v):v\in\mathrm T^{1,0}X\}.
\end{gathered}
$$
The integral Chern classes are constant under these identifications.
We use Kuranishi's theorem to realize sufficiently small
deformations in a semiuniversal deformation; see \cite{KS12,Kodaira05}.

\subsection{Bott--Chern and Aeppli cohomology}
Write $h^{p,q}(X)=\dim H^{p,q}_{\bar\partial}(X)$.
The Bott--Chern and Aeppli cohomology groups are
$$
\begin{aligned}
H^{p,q}_{BC}(X)
&=\frac{\ker\partial\cap\ker\bar\partial\cap A^{p,q}(X)}
{\operatorname{Im}\partial\bar\partial\cap A^{p,q}(X)},\\
H^{p,q}_A(X)
&=\frac{\ker\partial\bar\partial\cap A^{p,q}(X)}
{(\operatorname{Im}\partial+\operatorname{Im}\bar\partial)
\cap A^{p,q}(X)}.
\end{aligned}
$$
The $\partial\bar\partial$-lemma is equivalent to the injectivity of
the map $H^{p,q}_{BC}(X)\to H^{p,q}_A(X)$ for every $(p,q)$.
It gives
$$
H^n(X,\C)=\bigoplus_{p+q=n}H^{p,q}(X),\qquad
\overline{H^{p,q}(X)}=H^{q,p}(X),
$$
where $H^{p,q}(X)$ is the image of $H^{p,q}_{BC}(X)$ in de Rham
cohomology; see \cite{DGMS,AT13}.

For a fixed Hermitian metric, put
$$
\begin{aligned}
\triangle_{BC}={}&
(\partial\bar\partial)(\partial\bar\partial)^*
+(\partial\bar\partial)^*(\partial\bar\partial)\\
&+(\bar\partial^*\partial)(\bar\partial^*\partial)^*
+(\bar\partial^*\partial)^*(\bar\partial^*\partial)
+\bar\partial^*\bar\partial+\partial^*\partial.
\end{aligned}
$$
Let $\mathbb H_{BC}$ and $G_{BC}$ be its harmonic projection and
Green operator. The Bott--Chern Hodge decomposition
\cite[Section~1]{AT13} gives, for $\partial\alpha=\bar\partial\alpha=0$,
\begin{equation}\label{BC closed decomposition}
\alpha=\mathbb H_{BC}\alpha+
\partial\bar\partial(\partial\bar\partial)^*G_{BC}\alpha.
\end{equation}
We also use the Dolbeault Green operator $G_{\bar\partial}$ for
$\triangle_{\bar\partial}=\bar\partial\bar\partial^*
+\bar\partial^*\bar\partial$.

\section{The period map of a limit of compact hyperk\"ahler manifolds}
\label{ddbar HK limits section}
A holomorphic symplectic form is a closed, everywhere nondegenerate
holomorphic two-form. Let $X=X_0$ satisfy the hypotheses of
Theorem~\ref{main ddbar HK limit}. Choose an analytic family
realizing $X$ as a limit of compact hyperk\"ahler manifolds,
and put $d=\dim_{\C}X=2m$. The differentiable identifications give
$b_1(X)=0$ and $c_1(K_X)=0$ in $H^2(X,\mathbb Z)$.
Choose a Hermitian metric $\gamma$ on $X$ and a real closed
$(2d-2)$-form $\Psi$ such that
$$
\Psi^{d-1,d-1}=\gamma^{d-1},
$$
as provided by the strongly Gauduchon hypothesis and
Lemma~\ref{closed Gauduchon completion}.

The local Torelli theorem on the approaching compact hyperk\"ahler
fibers allows arbitrarily small deformations with trivial
N\'eron--Severi group. By Kuranishi's theorem, after
differentiable identifications tending to the identity,
we may realize these deformations in a Kuranishi family
$\pi:\X\to B$ of $X$, with
$$
t^i\longrightarrow0,\qquad X_i=X_{t^i},\qquad
\operatorname{NS}(X_i)=0.
$$
Use the fixed coordinate covering of Section~\ref{com section}
and choose Hermitian forms $\gamma_i\to\gamma$ in $C^\infty$.
Normalize holomorphic symplectic forms $\sigma_i$ by
$\|\sigma_i\|_{L^2(\gamma_i)}=1$.
Since they are Dolbeault harmonic, uniform elliptic estimates
give, after passing to a subsequence,
$$
\sigma_i\longrightarrow\sigma\quad\text{in }C^\infty,
\qquad \bar\partial\sigma=0,\quad d\sigma=0,\quad
\|\sigma\|_{L^2(\gamma)}=1.
$$
Fix a marking of $H^2(X,\mathbb Z)$ by $\Lambda$ and transport
the Beauville--Bogomolov form $q$ from the compact hyperk\"ahler
fibers to the central fiber. Then
$$
q([\sigma])=0,\qquad
q([\sigma],[\bar\sigma])
=\lim_i q([\sigma_i],[\bar\sigma_i])\ge0.
$$

Choose a three-dimensional
subspace $W\subset H^2(X,\mathbb R)$ with $q|_W>0$.
Since $H^{1,1}(X_i,\mathbb R)$ has codimension two, choose
$$
a_i\in W\cap H^{1,1}(X_i,\mathbb R),\qquad q(a_i)=1.
$$
These classes lie on the compact unit sphere of $(W,q)$.
As $\operatorname{NS}(X_i)=0$, a choice of sign makes $a_i$
K\"ahler. Let $\omega_{t^i}$ be its Ricci-flat representative
\cite{Yau78}. Lemma~\ref{closed Gauduchon completion} gives
\begin{equation}\label{h01 Ricci flat mass}
\int_{X_i}\operatorname{tr}_{\gamma_i}(\omega_{t^i})
                  \frac{\gamma_i^d}{d!}
=\frac1{(d-1)!}\int_{X_i}\omega_{t^i}\wedge\gamma_i^{d-1}
\le C.
\end{equation}
The form $\sigma_i$ is parallel for $\omega_{t^i}$.
The polarized Fujiki identity and a comparison in a
$\gamma_i$-unitary frame diagonalizing $\omega_{t^i}$ give
$$
\begin{aligned}
|\sigma_i|_{\omega_{t^i}}^2
&=c\,q([\sigma_i],[\bar\sigma_i]),\\
|\sigma_i|_{\gamma_i}
&\le C|\sigma_i|_{\omega_{t^i}}
              \operatorname{tr}_{\gamma_i}(\omega_{t^i}),
\end{aligned}
$$
with $c,C>0$ independent of $i$. Consequently,
\begin{equation}\label{h01 symplectic L1 estimate}
0<\|\sigma\|_{L^1(\gamma)}
=\lim_i\|\sigma_i\|_{L^1(\gamma_i)}
\le C\sqrt{q([\sigma],[\bar\sigma])}.
\end{equation}
The Fujiki identity now gives $\sigma^m\ne0$.

The form $\sigma^m$ is a nonzero section of $K_X$. Its zero divisor
satisfies
$$
\int_{\operatorname{div}(\sigma^m)}\Psi^{d-1,d-1}
=\int_{\operatorname{div}(\sigma^m)}\Psi
=\langle c_1(K_X),[\Psi]\rangle=0.
$$
Strict positivity of $\Psi^{d-1,d-1}$ forces
$\operatorname{div}(\sigma^m)=0$.
Thus $\sigma$ is a holomorphic symplectic form and
$\C[\sigma]$ belongs to the period domain
$$
D=\{\C\alpha\in\mathbb P(\Lambda\otimes\C):
q(\alpha)=0,\ q(\alpha,\bar\alpha)>0\}.
$$

The moduli space of marked compact hyperk\"ahler manifolds on the
underlying differentiable manifold has finitely many connected
components \cite[Theorem~2.6(ii)]{VerbitskyErratum}.
Pass to a subsequence in one component. Surjectivity of its
period map gives a marked compact hyperk\"ahler manifold $Y$ with
period $\C[\sigma]$. Let $\pi':\Y\to B'$ be its Kuranishi
family and $\Phi:B'\to D$ its period map. The local Torelli
theorem at $Y$ gives
$$
{t'}^i\longrightarrow0,\qquad
Y_i=Y_{{t'}^i},\qquad \Phi({t'}^i)=\C[\sigma_i].
$$

The resulting Hodge isometries are parallel-transport operators
and preserve the positive cones. Since both Picard groups vanish,
the K\"ahler cones equal the positive cones. The Torelli theorem
therefore gives biholomorphisms
\begin{equation}\label{HK marked comparison sequence}
f_i:X_i\xrightarrow{\sim}Y_i
\end{equation}
whose actions on second cohomology agree with the markings
\cite{HuybrechtsTorelli,Markman11,Verbitsky13,VerbitskyErratum}.
This is the period construction of \cite[Proposition~2.4]{Per}.

Choose K\"ahler forms $\omega_i\to\omega$ in $C^\infty$ on
$Y_i\to Y$. Normalize holomorphic symplectic forms
$\tau_i\to\tau$ so that their marked classes agree with
$[\sigma_i]$.
Since holomorphic forms on the compact hyperk\"ahler fibers inject
into de Rham cohomology,
\begin{equation}\label{HK symplectic comparison}
f_i^*\tau_i=\sigma_i.
\end{equation}
The classes $[f_i^*\omega_i]$ are bounded under the
Gauss--Manin identifications, and both limiting symplectic
forms are nondegenerate.

\section{Holomorphic maps of bounded Dirichlet energy}
\label{holomorphic maps section}

Let $f:(X,\gamma)\to(Y,\omega)$ be a smooth map between
Hermitian manifolds, with $\dim_{\C}X=d$. Its Dirichlet energy
on an open subset $V\subset X$ is
$$
E(f;V)=\frac12\int_V|df|_{\gamma,\omega}^2\,
                         \frac{\gamma^d}{d!}.
$$
Here $|df|_{\gamma,\omega}^2$ is the sum of the squared lengths
of $df(e_j)$ over a real orthonormal frame $(e_j)$ for the
Riemannian metric associated to $\gamma$, with lengths on $Y$
measured by the metric associated to $\omega$. We write
$E(f)=E(f;X)$. If $f$ is holomorphic, a pointwise calculation
in a unitary frame gives
\begin{equation}\label{holomorphic Dirichlet energy}
E(f;V)=\frac1{(d-1)!}\int_V f^*\omega\wedge\gamma^{d-1}.
\end{equation}
This identity holds for any Hermitian metric $\gamma$ on $X$.
It relates the energy of a holomorphic map to the mass of the
positive $(1,1)$-form $f^*\omega$.

\begin{lemma}\label{closed Gauduchon completion}
Let $X$ be a compact complex manifold of dimension $d\ge2$.
Then the following hold.
\begin{enumerate}
\item[(i)] The manifold $X$ is strongly Gauduchon if and only if
there is a Hermitian metric $\gamma$ such that
$\partial\gamma^{d-1}$ is $\bar\partial$-exact. Such a metric
admits a real $(2d-2)$-form $\Psi$ satisfying
$$
d\Psi=0,\qquad \Psi^{d-1,d-1}=\gamma^{d-1}.
$$
If $X$ satisfies the $\partial\bar\partial$-lemma or
$h^{0,1}(X)=0$, every Gauduchon metric admits such a completion;
in particular, $X$ is strongly Gauduchon.
\item[(ii)] Suppose that $X$ is strongly Gauduchon.
Let $\pi:\mathcal X\to B$ be a deformation of $X=X_0$,
and put $X_i=X_{t^i}$ for $t^i\to0$. Let $\gamma_i$ be Hermitian
forms on $X_i$ converging in $C^\infty$ to a Hermitian form on $X$.
If $\alpha_i$ are closed positive $(1,1)$-currents on $X_i$ whose
de Rham classes are bounded under the Gauss--Manin identification
with $H^2(X,\mathbb R)$, then
$$
\int_{X_i}\alpha_i\wedge\gamma_i^{d-1}\le C.
$$
\end{enumerate}
\end{lemma}
\begin{proof}
For (i), if $\bar\partial v=\partial\gamma^{d-1}$ with
$v\in A^{d,d-2}(X)$, then
$$
\Psi=\gamma^{d-1}-v-\bar v,\qquad
d\Psi=0,\qquad \Psi^{d-1,d-1}=\gamma^{d-1}.
$$
Conversely, the strictly positive component $\Psi^{d-1,d-1}$
of a real closed form has a Hermitian $(d-1)$-st root $\gamma$,
and $d\Psi=0$ gives
$\partial\gamma^{d-1}=-\bar\partial\Psi^{d,d-2}$
\cite[Lemma~1.1]{PopoviciSG}.

Now let $\gamma$ be any Gauduchon metric. Under the
$\partial\bar\partial$-lemma,
$$
\partial\gamma^{d-1}\in\ker d\cap\operatorname{Im}\partial
=\operatorname{Im}\partial\bar\partial
\quad\text{in bidegree }(d,d-1).
$$
By \eqref{BC closed decomposition},
$$
\begin{gathered}
u=(\partial\bar\partial)^*G_{BC}(\partial\gamma^{d-1}),
\qquad \partial\bar\partial u=\partial\gamma^{d-1},\\
\Psi=\gamma^{d-1}+\partial u+\overline{\partial u},\qquad
 d\Psi=0,\qquad \Psi^{d-1,d-1}=\gamma^{d-1}.
\end{gathered}
$$
If $h^{0,1}(X)=0$, Serre duality gives
$H^{d,d-1}_{\bar\partial}(X)=0$. Set
$$
\begin{gathered}
v=\bar\partial^*G_{\bar\partial}(\partial\gamma^{d-1}),
\qquad \bar\partial v=\partial\gamma^{d-1},\\
\Psi=\gamma^{d-1}-v-\bar v,\qquad
 d\Psi=0,\qquad \Psi^{d-1,d-1}=\gamma^{d-1}.
\end{gathered}
$$

For (ii), choose one real closed form $\Psi$ with
$\Psi^{d-1,d-1}>0$. Let $\Psi_i$ be the closed form on $X_i$ obtained
from $\Psi$ by the chosen differentiable identification.
Its $(d-1,d-1)$ component on $X_i$ remains strictly positive.
For some fixed $c>0$ and all sufficiently large $i$,
$\Psi_i^{d-1,d-1}\ge c\gamma_i^{d-1}$. Therefore
$$
\int_{X_i}\alpha_i\wedge\gamma_i^{d-1}
\le c^{-1}\int_{X_i}\alpha_i\wedge\Psi_i
=c^{-1}\langle[\alpha_i],[\Psi]\rangle\le C.
$$
Here the cohomology classes are compared under these identifications.
\end{proof}

For $\alpha_i=f_i^*\omega_i$, the lemma gives $\sup_i E(f_i)<\infty$.
The volume of the graph $\Gamma_i\subset X_i\times Y_i$ of $f_i$ with
respect to the product of the forms $\gamma_i$ and $\omega_i$ is
$$
\operatorname{Vol}(\Gamma_i)
=\frac1{d!}\int_{X_i}(\gamma_i+\alpha_i)^d
=\frac1{d!}\sum_{k=0}^d\binom dk
  \int_{X_i}\gamma_i^{d-k}\wedge\alpha_i^k.
$$
The terms $k=0,d$ are bounded; Lemma~\ref{closed Gauduchon completion}
bounds the term $k=1$. For $d=2$, Bishop's theorem applies as in
\cite[Theorem~4.3]{Hu1}. For $d>2$ it also requires the terms
$2\le k\le d-1$, which the preceding estimate does not control;
see \cite[Section~1.2]{SV24} for the use of Bishop's theorem in
\cite[Lemma~2.5]{Per}.
The following results obtain a meromorphic limit from the energy bound.

\begin{lemma}\label{finite energy meromorphic extension}
Let $(U,\gamma)$ be a Hermitian manifold of dimension $d\ge2$, let
$A\subset U$ be an analytic subset of positive codimension, and let
$(Y,\omega)$ be a compact K\"ahler manifold.
A holomorphic map $f:U\setminus A\to Y$ with
$E(f;V\setminus A)<\infty$ for every relatively compact open
subset $V\Subset U$ extends meromorphically to $U$.
\end{lemma}
\begin{proof}
Work in a relatively compact coordinate polydisc
$P\simeq\Delta^d$ in $U$. Replacing $\gamma$ by the Euclidean
metric does not affect finiteness of the energy on $P$.
For each coordinate direction, Fubini's theorem shows that
$f|_{D\setminus A}$ has finite energy for almost every
coordinate disk $D\subset P$. Since $A$ has measure zero,
almost every such disk is not contained in $A$ and hence
meets it in a discrete set. Since holomorphic maps from a Riemann surface
to a K\"ahler manifold are harmonic, the removable-singularity
theorem of Sacks and Uhlenbeck
\cite[Theorem~3.6]{SacksUhlenbeck81} extends this restriction
holomorphically across every point of $D\cap A$.

The set $P\setminus A$ has full measure in $P$, and the
preceding extensions exist for almost every coordinate disk
in each direction. Shiffman's theorem
\cite[Theorem~1, p.~244]{Shiffman94} therefore gives a
meromorphic map $P\dashrightarrow Y$ agreeing with $f$
almost everywhere on $P\setminus A$. It agrees there as a
meromorphic map by uniqueness. The local extensions
consequently glue to a meromorphic extension on $U$.
\end{proof}

\begin{proposition}\label{finite energy compactness}
Let $\pi:\X\to B$ and $\pi':\Y\to B'$ be deformations of the compact
complex manifolds $X=X_0$ and $Y=Y_0$, respectively, with
$\dim_{\C}X=d\ge2$. For $t^i\to0$ and ${t'}^i\to0$, put
$X_i=X_{t^i}$ and $Y_i=Y_{{t'}^i}$.
Fix differentiable trivializations. Let $\gamma_i$ be Hermitian
forms on $X_i$ and $\omega_i$ K\"ahler forms on $Y_i$ such that
$$
\gamma_i\longrightarrow\gamma,\qquad
\omega_i\longrightarrow\omega
\quad\text{in }C^\infty,
$$
where $\gamma$ is a Hermitian form on $X$ and $\omega$ is a
K\"ahler form on $Y$. Suppose that holomorphic maps
$f_i:X_i\to Y_i$ satisfy
$$
\sup_i\int_{X_i}f_i^*\omega_i\wedge\gamma_i^{d-1}<\infty.
$$
Then there exist a proper analytic subset $A\subset X$ and a
meromorphic map $f:X\dashrightarrow Y$, holomorphic on
$X\setminus A$, such that, after passing to a subsequence,
$$
f_i\longrightarrow f
\quad\text{in }C^\infty_{\mathrm{loc}}(X\setminus A,Y).
$$
\end{proposition}
\begin{proof}
The Beltrami differentials of $X_i$ relative to $X$ satisfy
$\phi(t^i)\to0$ in $C^\infty$. The mass bound therefore gives,
after passing to a subsequence, a current $T$ on $X$ satisfying
$$
\begin{gathered}
f_i^*\omega_i\rightharpoonup T,\qquad
 dT=\lim_i d(f_i^*\omega_i)=0,\\
T^{2,0}=T^{0,2}=0,\qquad T\ge0.
\end{gathered}
$$

Fix the covering
$\mathscr U=\{(V_\alpha\subset U_\alpha;z_\alpha)\}$
as in Section~\ref{com section}. On the same open sets, choose
holomorphic coordinates $\zeta_{t^i,\alpha}$ for $X_i$ with
$\zeta_{t^i,\alpha}\to z_\alpha$ in $C^\infty$.
In each chart, write $f_i$ for
$f_i\circ\zeta_{t^i,\alpha}^{-1}$ and set
$\beta=\frac{\sqrt{-1}}2\sum_{j=1}^d dz_j\wedge d\bar z_j$.
All balls below lie in a common relatively compact subset of
the coordinate images. Norms and integrals use the Euclidean
source metric and the target metric $\omega_i$.
The maps $f_i$ are harmonic in these coordinates. 

We claim
that there are $\varepsilon_0,C>0$, independent of $i$, such that
\begin{equation}\label{holomorphic small energy estimate}
r^{2-2d}\int_{B_r}|df_i|^2<\varepsilon_0
\quad\Longrightarrow\quad
\sup_{B_{r/2}}|df_i|^2
\le Cr^{-2d}\int_{B_r}|df_i|^2.
\end{equation}

Rescale to $r=1$. The Bochner inequality and energy
monotonicity give
$$
\begin{gathered}
\Delta|df_i|^2\ge-C|df_i|^4,\\
\rho^{2-2d}\int_{B_\rho(x)}|df_i|^2
\le4^{2d-2}\int_{B_1}|df_i|^2,
\quad x\in B_{3/4},\quad 0<\rho\le\tfrac14.
\end{gathered}
$$
The constant in the Bochner inequality is uniform because
the target metrics converge smoothly.
Choose $x_i$ maximizing $(3/4-|x|)^2|df_i(x)|^2$.
If this maximum exceeds $16$, put
$$
\lambda_i=|df_i(x_i)|^{-1}<\tfrac14,\qquad
\widetilde f_i(z)=f_i(x_i+\lambda_i z),\quad z\in B_1.
$$
Then
$$
|d\widetilde f_i(0)|^2=1,\qquad
\sup_{B_1}|d\widetilde f_i|^2\le4,\qquad
\Delta|d\widetilde f_i|^2\ge-4C|d\widetilde f_i|^2.
$$
The mean-value inequality and energy monotonicity yield
$$
\begin{aligned}
1&\le C\int_{B_1}|d\widetilde f_i|^2
 =C\lambda_i^{2-2d}
       \int_{B_{\lambda_i}(x_i)}|df_i|^2\\
 &\le C\int_{B_1}|df_i|^2<C\varepsilon_0,
\end{aligned}
$$
a contradiction for $\varepsilon_0$ sufficiently small.

Hence $\sup_{B_{5/8}}|df_i|^2\le C$. Applying the
mean-value inequality to $\Delta|df_i|^2\ge-C|df_i|^2$ on
$B_{5/8}$ gives
$$
\sup_{B_{1/2}}|df_i|^2\le C\int_{B_1}|df_i|^2,
$$
and scaling back proves \eqref{holomorphic small energy estimate}.

For almost every sufficiently small $r$, weak convergence and
\eqref{holomorphic Dirichlet energy} give
$$
\lim_{i\to\infty}r^{2-2d}\int_{B_r(x)}|df_i|^2
=\frac{2r^{2-2d}}{(d-1)!}
  \int_{B_r(x)}T\wedge\beta^{d-1}
\xrightarrow{r\downarrow0}c_d\nu(T,x),
$$
where $c_d>0$ depends only on the dimension.
Here the excluded radii satisfy
$(T\wedge\beta^{d-1})(\partial B_r(x))>0$.
By Siu's theorem \cite{Siu74},
$$
A=\{x\in X:\nu(T,x)\ge\varepsilon_0/(2c_d)\}
$$
is a proper analytic subset. For $x\notin A$, choose $r>0$
with
$$
\lim_{i\to\infty}r^{2-2d}\int_{B_r(x)}|df_i|^2<\varepsilon_0.
$$
Estimate~\eqref{holomorphic small energy estimate}, elliptic
regularity and a diagonal argument give a subsequence with
$$
f_i\longrightarrow f
\quad\text{in }C^\infty_{\mathrm{loc}}(X\setminus A,Y).
$$
The convergence of the complex tangent subbundles gives
$$
df_i(\aTan X_i)\subset\aTan Y_i
\quad\Longrightarrow\quad
df(\aTan X)\subset\aTan Y.
$$
Thus $f$ is holomorphic. Moreover,
$$
\begin{aligned}
E(f;X\setminus A)
&=\frac1{(d-1)!}\int_{X\setminus A}
                    f^*\omega\wedge\gamma^{d-1}\\
&\le\liminf_{i\to\infty}\frac1{(d-1)!}
             \int_{X_i}f_i^*\omega_i\wedge\gamma_i^{d-1}<\infty.
\end{aligned}
$$
Lemma~\ref{finite energy meromorphic extension} extends $f$
meromorphically to $X$.
\end{proof}

\section{Proof of the limit theorem for a degeneration of hyperk\"ahler
manifolds}
\label{HK limit proof section}
\begin{proof}[Proof of Theorem~\ref{main ddbar HK limit}]
Use the real closed form $\Psi$ chosen in
Section~\ref{ddbar HK limits section}, whose $(2m-1,2m-1)$ component
is strictly positive.
For Hermitian metrics $\gamma_i$ on $X_i$ converging smoothly
to a Hermitian metric on $X$, the bounded marked classes of
\eqref{HK marked comparison sequence} therefore satisfy
\begin{equation}\label{HK comparison energy}
\int_{X_i}f_i^*\omega_i\wedge\gamma_i^{2m-1}
\le C\langle[f_i^*\omega_i],[\Psi]\rangle\le C'.
\end{equation}
Proposition~\ref{finite energy compactness} gives a meromorphic
map $f:X\dashrightarrow Y$ and smooth convergence of a subsequence
of $f_i$ outside a proper analytic subset $A\subset X$.
Passing to the limit in \eqref{HK symplectic comparison} gives
$f^*\tau=\sigma$ there. Nondegeneracy of $\sigma$ makes $f$
dominant and generically finite.

Its degree is one. Otherwise a general point of $Y$ would have
two distinct preimages $x,y\in X\setminus A$ at which $f$ is
locally biholomorphic. Indeed, on a resolution of the graph,
the images of the exceptional locus and of the inverse image
of $A$ are proper analytic subsets of $Y$.
Choose disjoint neighborhoods of $x,y$ whose images contain
a common neighborhood of $f(x)=f(y)$. Smooth convergence and
stability of local degree imply that their images under $f_i$
intersect for all large $i$, contrary to injectivity of $f_i$.

A dominant meromorphic map of generic degree one between compact
complex manifolds of the same dimension is bimeromorphic.
Thus $f$ is bimeromorphic. Since $Y$ is K\"ahler, $X$ belongs
to the Fujiki class $\mathcal C$.
\end{proof}

The central fiber need not be K\"ahler: non-K\"ahler Mukai
transforms can occur in such families
\cite[Example~21.9]{GHJ}.

\section{The \texorpdfstring{$\partial\bar\partial$}{ddbar}-lemma for
the moduli space of sheaves on a K3 surface}
\label{moduli Kahler section}
Let $S$ be a K3 surface and let $v=(r,\xi,a)$ be a Mukai
vector with $r>0$ and $r,\xi$ relatively prime. This means that no
integer greater than one divides both $r$ and $\xi$ in
$\operatorname{NS}(S)$. Let $\omega$ be a $v$-generic K\"ahler class,
and suppose that the moduli space $M=M_v(S,\omega)$ of
$\mu_\omega$-stable torsion-free coherent sheaves on $S$ is nonempty and
has positive dimension. Perego--Toma \cite{PT} prove that $M$ is a
compact smooth holomorphically symplectic manifold, deformation
equivalent to a Hilbert scheme $S^{[n]}$ of points on a K3 surface,
where $2n=v^2+2$. We prove that $M$ is K\"ahler.

We prove a Bott--Chern refinement of Markman's diagonal formula
\cite{MarkmanGenerators}, and deduce the
$\partial\bar\partial$-lemma.

We use the Bott--Chern Chern character $\operatorname{ch}$ of
Bismut--Shen--Wei \cite{BSW23} for bounded
complexes with coherent cohomology,
which is additive and compatible with derived pullback and tensor
products
\cite[Theorems~8.14--8.16]{BSW23}, agrees with the usual Chern character
of a holomorphic vector bundle, and satisfies
Grothendieck--Riemann--Roch for proper holomorphic maps
\cite[Theorem~1.1]{BSW23};
$$
c(K)=\exp\left(\sum_{k\geq1}(-1)^{k-1}(k-1)!
                   \operatorname{ch}_k(K)\right),
\qquad
\operatorname{ch}_k(K)=0\ \ (k>\dim X).
$$
Pushforwards may be computed with currents: forms and currents give the
same Bott--Chern cohomology, with compatible products and pushforwards
\cite[Sections~2.2--2.3]{BSW23}.

\begin{lemma}\label{BC diagonal criterion}
Let $M$ be a compact complex manifold. Suppose that its diagonal has a
decomposition
\begin{equation}\label{BC diagonal decomposition}
[\Delta_M]_{BC}=\sum_j a_j\boxtimes b_j,
\end{equation}
where $a_j$ and $b_j$ are Bott--Chern classes of pure bidegree on $M$.
Then $M$ satisfies the $\partial\bar\partial$-lemma.
\end{lemma}

\begin{proof}
Let $u\in H^{p,q}_{BC}(M)$ have zero Aeppli class, represented
by $\partial\alpha+\bar\partial\beta$. Choose $d$-closed
representatives of the $a_j$, denoted by the same symbols. Since the diagonal acts
as the identity, \eqref{BC diagonal decomposition} and Stokes' theorem give
$$
u=\sum_j\pm\left(\int_M(\partial\alpha+\bar\partial\beta)
\wedge a_j\right)b_j=0.
$$
Thus $H^{p,q}_{BC}(M)\to H^{p,q}_A(M)$ is injective in every
bidegree, and $M$ satisfies the $\partial\bar\partial$-lemma
\cite{DGMS,AT13}.
\end{proof}

We will also use the following consequence of the
$\partial\bar\partial$-lemma on a compact K\"ahler manifold $S$:
\begin{equation}\label{BC Kunneth with Kahler factor}
H^{p,q}_{BC}(S\times P)
 =\bigoplus_{a+c=p,\ b+d=q}
       H^{a,b}(S)\otimes H^{c,d}_{BC}(P),
\qquad
H^{a,b}(S)=H^{a,b}_{\bar\partial}(S),
\end{equation}
for every compact complex manifold $P$, the isomorphism being induced by
exterior products of $d$-closed forms. Indeed, since $S$ is K\"ahler the
inclusion of its harmonic representatives is an $E_1$-isomorphism of
double complexes; tensor products and the exterior-product map to the
double complex of $S\times P$ preserve $E_1$-isomorphisms, and these
induce isomorphisms on Bott--Chern cohomology
\cite[Corollary~13, Lemma~15 and Section~4]{StelzigDoubleComplexes}.

\begin{lemma}\label{analytic Ext diagonal formula}
Let $M$ be a compact smooth moduli space of stable coherent sheaves
on a K3 surface $S$, with a fixed Mukai vector
and polarization, of even dimension $m\geq4$. Suppose that
$p:P\to M$ is a holomorphic projective-space bundle carrying a universal
family $\mathcal U$ on $S\times P$: the sheaf parametrized by $x\in P$
represents $p(x)\in M$. Put
$$
Z=P\times_M P,\qquad j:Z\hookrightarrow P\times P.
$$
Let $\pi_1,\pi_2:S\times P\times P\to S\times P$ be the two projections,
and let $\pi:S\times P\times P\to P\times P$. Put
$$
K=R\pi_*R\mathcal Hom
       (\pi_1^{*}\mathcal U,\pi_2^{*}\mathcal U)[1].
$$
Then
\begin{equation}\label{analytic BC Ext top Chern}
c_m(K)=[Z]_{BC}.
\end{equation}
The assertion includes $P=M$ when a universal sheaf exists on $S\times M$.
\end{lemma}

\begin{proof}
Fix a nonzero holomorphic two-form on $S$. The complex $K$ is
perfect, with Tor amplitude $[-1,1]$ and virtual rank $m-2$.
Stability and Serre duality give vanishing of the fiberwise
Hom and $\operatorname{Ext}^2$ groups away from $Z$.
Let $L$ be the Hom line on $Z$. The trace pairing identifies
it with the $\operatorname{Ext}^2$ line. If $\tau$ exchanges the
two factors, relative Serre duality gives
\begin{equation}\label{Ext swap duality}
K^\vee\simeq\tau^*K,\qquad \tau^*L\simeq L^\vee.
\end{equation}

We recall the local calculation underlying the diagonal formula;
it is \cite[Section~4]{MarkmanBB}, uses only analytic neighborhoods,
and carries over verbatim to the relative setting as in the proof of
\cite[Proposition~4.2]{MarkmanStability}. Near a point of $Z$, after
removing contractible summands and trivializing the endpoint lines, a
locally free representative of $K$ is
\begin{equation}\label{local Ext complex}
\mathcal O\xrightarrow{g}\mathcal O^m\xrightarrow{f}\mathcal O,
\qquad f\circ g=0,
\end{equation}
with both maps vanishing on $Z$. For first-order deformations
$\alpha,\beta\in\operatorname{Ext}^1(F,F)$ of the two sheaves, the
derivatives of the maps on $\operatorname{id}_F$ and
$e\in\operatorname{Ext}^1(F,F)$ are, respectively,
$$
D g(\alpha,\beta)=\beta-\alpha,\qquad
D f(\alpha,\beta)=\operatorname{tr}((\beta-\alpha)\circ e).
$$
The first is the Kodaira--Spencer isomorphism on the normal
space of $Z$, and the second is its dual for the nondegenerate
alternating Ext pairing. Hence the entries of $f$ generate
$\mathcal I_Z/\mathcal I_Z^2$, and therefore $\mathcal I_Z$ by
Nakayama's lemma, while $\mathcal H^{-1}(K)=0$: a kernel supported on
$Z$ would be a torsion subsheaf of a locally free sheaf. Thus
$$
\mathcal H^1(K)=j_*L,
$$
and $L$ is a line on the reduced $Z$. Directions tangent to the
fibers of $P\to M$ do not change the normal calculation.

Let $b:B=\operatorname{Bl}_Z(P\times P)\to P\times P$, with exceptional
divisor $i:D\hookrightarrow B$ and projection $p_D:D\to Z$. Write
$Q=Lb^*K$, $H=\mathcal H^0(Q)$ and $N=\mathcal O_D(D)$, $L_D=p_D^*L$.
Dividing the two maps in \eqref{local Ext complex} by an equation of $D$
gives a complex
$$
\mathcal O_B(D)\longrightarrow\mathcal O_B^m
                        \longrightarrow\mathcal O_B(-D)
$$
exact at both ends, so
$$
V=H/\operatorname{tors}(H)
$$
is locally free of rank $m-2$. On a fiber $F\simeq\mathbb P^{m-1}$ of
$p_D$ it is the quotient $\ell_F^\perp/\ell_F$, where
$\ell_F=\mathcal O_F(-1)\subset
\operatorname{Ext}^1(F_0,F_0)\otimes\mathcal O_F$ is the tautological
line and orthogonality is taken for the alternating form above; here
$F_0$ is the sheaf represented by the image of $F$ in $M$. In particular,
if $h=c_1(\mathcal O_F(1))$, then
\begin{equation}\label{exceptional fiber Chern}
c(V|_F)=\frac1{(1-h)(1+h)}.
\end{equation}
The torsion term of $H$ is identified globally, without a global locally
free resolution of $K$, by dualizing the truncation triangle
$H\to Q\to i_*L_D[-1]$ with
$R\mathcal Hom(i_*L_D,\mathcal O_B)=i_*(L_D^\vee\otimes N)[-1]$: this gives
$$
0\longrightarrow i_*(L_D^\vee\otimes N)
 \longrightarrow\widetilde\tau^*H\longrightarrow V^\vee
 \longrightarrow0,
$$
$\widetilde\tau$ the lift of $\tau$ to $B$, whose first term is exactly
the torsion subsheaf. Applying $\widetilde\tau$ and \eqref{Ext swap duality}
therefore gives
\begin{equation}\label{torsion and quotient classes}
\operatorname{tors}(H)=i_*(L_D\otimes N),\qquad
[Q]=[V]+[i_*(L_D\otimes N)]-[i_*L_D].
\end{equation}

Set $d=c_1(N)$ and $l=c_1(L_D)$. The
Grothendieck--Riemann--Roch formula for the divisor $i$
\cite[Theorem~9.2]{BSW23}, together with the self-intersection identity
$i_*u\,i_*v=i_*(uvd)$ and the projection formula of
\cite[Propositions~6 and~11]{WuIntersection}, gives
$$
\operatorname{ch}\bigl(i_*(L_D\otimes N)-i_*L_D\bigr)
 =i_*\left(e^l\frac{e^d+e^{-d}-2}{d}\right),
$$
and the formal identity
$(1+l+D)(1+l-D)/(1+l)^2=1-D^2/(1+l)^2$, read through
$i_*u\,i_*v=i_*(uvd)$ intrinsically on $D$ and so requiring no extension
of $L_D$ to $B$, yields
\begin{equation}\label{exceptional correction Chern}
c\bigl(i_*(L_D\otimes N)-i_*L_D\bigr)
 =1-i_*\left(\frac{d}{(1+l)^2}\right).
\end{equation}
Since $\deg V=m-2$, derived pullback compatibility and
$b_*b^*=\operatorname{id}$ convert \eqref{torsion and quotient classes}
and \eqref{exceptional correction Chern} into
\begin{equation}\label{top Chern exceptional push}
c_m(K)=-j_*(p_D)_*
 \left[c(V|_D)\frac{d}{(1+l)^2}\right]_{m-1},
\end{equation}
the subscript denoting the component of Chern degree $m-1$. Every
positive power of $l$ is pulled back from $Z$ and contributes zero to
this fiber integral. On a fiber $F$ one has $d|_F=-h$ and, by
\eqref{exceptional fiber Chern},
$c_{m-2}(V|_F)=h^{m-2}$ because $m$ is even. The coefficient of $j_*1$
in \eqref{top Chern exceptional push} is therefore
$$
-\int_{\mathbb P^{m-1}}(-h)h^{m-2}=1,
$$
which is \eqref{analytic BC Ext top Chern}.
\end{proof}

\begin{theorem}\label{moduli ddbar from diagonal}
Let $S$ be a K3 surface, let $v=(r,\xi,a)$ satisfy $r>1$,
$\gcd(r,\xi)=1$ and $v^2\geq0$, and let $\omega$ be a $v$-generic
K\"ahler class. Assume that the moduli space
$M=M_v(S,\omega)$ is nonempty. Then $M$ satisfies the
$\partial\bar\partial$-lemma.
\end{theorem}

\begin{proof}
By \cite[Theorem~1.1]{PT}, $M$ is a compact smooth holomorphically
symplectic manifold of dimension $m=v^2+2$. The case $v^2=0$, in which
$M$ is a K3 surface, is known; assume $m\geq4$.

If $S$ is projective, the result follows from
\cite[Theorem~3.4]{PT}. Suppose that $S$ is nonprojective.
By \cite[Proposition~2.3]{PT}, choose an irreducible locally free
sheaf $E_0$ of rank $r$, first Chern class $\xi$, and sufficiently
large second Chern class, with Mukai vector $(r,\xi,a')$ and $a'<a$.
It is $\mu_\omega$-stable, and for every sheaf $F$ parametrized by $M$,
\begin{equation}\label{no maps to E0}
\operatorname{Hom}(F,E_0)=\operatorname{Hom}(E_0,F)=0:
\end{equation}
a nonzero map has full rank, and the first direction forces $a'\ge a$;
in the second, local freeness identifies $E_0$ with $F^{**}$ and again
forces $a'\ge a$.

For local universal families $\mathcal U_\alpha$, \eqref{no maps to E0},
Grauert's theorem and Serre duality give locally free sheaves
$$
W_\alpha=R^1q_{\alpha*}
  (\mathcal U_\alpha\otimes\operatorname{pr}_S^*E_0^\vee)
$$
of rank $\rho=\xi^2-r(a+a')>0$, where $q_\alpha$ is the projection
to the parameter space. They have the same scalar descent
obstruction as $\mathcal U_\alpha$, as in the proof of
\cite[Proposition~4.19]{PT}. Their bundles of lines
$\mathbb P(W_\alpha)$ glue to a Brauer--Severi bundle $p:P\to M$.
Twisting the pullback of each
$\mathcal U_\alpha$ by the dual tautological line removes the scalar
obstruction and gives a universal sheaf $\mathcal U$ on
$S\times P$. This is the same construction used in
\cite[Section~3]{MarkmanGenerators}; no ample line bundle on $S$ is
required here.

For the complex $K$ of Lemma~\ref{analytic Ext diagonal formula},
Grothendieck--Riemann--Roch for the projection
$\pi:S\times P\times P\to P\times P$
\cite[Theorem~10.1]{BSW23} gives
\begin{equation}\label{BC universal kernel GRR}
\operatorname{ch}(K)
 =-\pi_*\left(
      \operatorname{ch}(\pi_1^*\mathcal U^\vee)
      \operatorname{ch}(\pi_2^*\mathcal U)
      \pi_S^*\operatorname{td}(S)\right),
\end{equation}
$\pi_S:S\times P\times P\to S$ the projection and the dual the derived
dual. Writing $\operatorname{ch}(\mathcal U)$ and
$\operatorname{ch}(\mathcal U^\vee)$ as finite sums of exterior products
of pure $d$-closed classes on $S$ and $P$ by
\eqref{BC Kunneth with Kahler factor}, integration over $S$ in
\eqref{BC universal kernel GRR} expresses each component of
$\operatorname{ch}(K)$, hence $c_m(K)$, as a finite sum of exterior
products of Bott--Chern classes on $P$. Lemma
\ref{analytic Ext diagonal formula} gives
$$
[P\times_M P]_{BC}=\sum_j u_j\boxtimes v_j .
$$

Put
$$
\zeta=\frac1\rho c_1(K_{P/M}^{-1}),
\qquad
p_*\zeta^{\rho-1}=1,
$$
the second identity because $\zeta|_F$ is the hyperplane class on a
fiber $\mathbb P^{\rho-1}$. Let $\zeta_i$ be its pullback from
the $i$th factor of $P\times P$ and apply
$(p\times p)_*$ to the preceding identity after multiplying by
$\zeta_1^{\rho-1}\zeta_2^{\rho-1}$. The result is
$$
[\Delta_M]_{BC}
 =\sum_j p_*(u_j\zeta^{\rho-1})
             \boxtimes p_*(v_j\zeta^{\rho-1}),
$$
with $d$-closed factors of pure bidegree. Lemma \ref{BC diagonal
criterion} concludes the proof.
\end{proof}

\section{Positivity of the Donaldson class on curves in the moduli
space}
\label{Donaldson positivity section}
Let $S$ be a K3 surface, let $v=(r,\xi,a)$ be a Mukai vector with
$r>0$ and $(r,\xi)=1$, and let $\omega$ be a $v$-generic K\"ahler
class on $S$.
Suppose that $M=M_v(S,\omega)$ is nonempty and $v^2>0$.
Write $\dim M=2n=v^2+2$. We use the Mukai pairing
$$
 ((r,\xi,a),(s,\zeta,b))=\xi\cdot\zeta-rb-sa
$$
and the convention
$$
 \lambda_v(w)=\frac1\rho
 \left[p_{M*}\left(\operatorname{ch}(\mathcal F)
 p_S^*\bigl(w^\vee\sqrt{\operatorname{td}(S)}\bigr)\right)\right]_2
$$
for a quasi-universal family $\mathcal F$ of similitude $\rho$.
Here $p_S,p_M$ are the projections from $S\times M$,
$w^\vee=(s,-\zeta,b)$ for $w=(s,\zeta,b)$, and $[\,\cdot\,]_2$
denotes the component of cohomological degree two.
We also write $\lambda_v$ for its $\mathbb R$-linear extension.
Let $q$ be the Beauville--Bogomolov form on $H^2(M,\mathbb Z)$,
transported from that of a Hilbert scheme $S^{[n]}$ of points on a K3
surface. The Mukai homomorphism is an isometry
\cite[Proposition~5.2]{PT}. Put
\begin{equation}\label{eq:moduli-donaldson-classes}
 A=\lambda_v(0,r\omega,\omega\cdot\xi),\qquad
 B=\lambda_v(-r,0,a).
\end{equation}
Both arguments belong to $v^\perp\otimes\mathbb R$, and $q(A)=r^2\omega^2>0$.
In particular, $mA+B$ corresponds to
$(-r,mr\omega,a+m\omega\cdot\xi)$.

We shall use the discriminant
$$
 \Delta(\mathcal E)=c_1(\mathcal E)^2-2r\operatorname{ch}_2(\mathcal E)
 =2r c_2(\mathcal E)-(r-1)c_1(\mathcal E)^2.
$$
For a vector bundle with a Hermitian metric, the same notation denotes
the Chern--Weil form. Thus
$\Delta(\mathcal E,H)=-\operatorname{ch}_2(\operatorname{End}\mathcal E,H)$.

\begin{lemma}\label{lem:curvewise-donaldson-degree}
Let $C$ be a compact smooth curve and let $\mathcal E$ be a reflexive
sheaf on $S\times C$, flat over $C$. Suppose that its fibers are
slope-stable torsion-free coherent sheaves of rank $r$ with the same
determinant, and that its general fiber is locally free. Then
$$
 \frac12\int_{S\times C}\Delta(\mathcal E)\wedge p_S^*\omega\geq0.
$$
The inequality is strict if the induced map from $C$ to the moduli
space $M$ is nonconstant.
\end{lemma}

\begin{proof}
Choose a K\"ahler form $\omega_S$ on $S$ representing $\omega$.
A reflexive sheaf on the smooth threefold $S\times C$ is locally free
away from finitely many points. Choose a sequence of blowups
$\pi:Z\longrightarrow S\times C$ which is an isomorphism off the
singular set of $\mathcal E$ and such that
$\widetilde{\mathcal E}=\pi^*\mathcal E/\mathrm{torsion}$ is locally
free, and put $h=p_C\circ\pi$. The manifold $Z$ is compact K\"ahler.
For a smooth Hermitian metric $H_0$ on $\widetilde{\mathcal E}$, define
the closed real $(1,1)$-current
$$
 R=\frac12h_*\left(\Delta(\widetilde{\mathcal E},H_0)
                  \wedge\pi^*p_S^*\omega_S\right).
$$
The pushforward $\pi_*\Delta(\widetilde{\mathcal E})$ agrees with
$\Delta(\mathcal E)$ away from the finitely many singular points of
$\mathcal E$, and local cohomology with this support vanishes in degree
four. Consequently
\begin{equation}\label{eq:curvewise-initial-current-class}
 [R]=\frac12p_{C*}\left(\Delta(\mathcal E)\wedge p_S^*\omega\right).
\end{equation}

The map $h$ is flat and open: a local parameter on the smooth curve $C$
is a nonzerodivisor on the smooth irreducible manifold $Z$. If
$\omega_Z$ is a K\"ahler form on $Z$, then, for some $K>0$,
$$
 -K h_*(\omega_Z^3)\leq R\leq K h_*(\omega_Z^3),
$$
and Varouchas' theorem on fiber integrals, applied as in
\cite[\S5.2]{BiswasSchumacherWP}, gives bounded continuous local
potentials for $R$.

Let $C^*\subset C$ be the complement of the finitely many exceptional
parameters. On $S\times C^*$ the metric $H_0$ is a reference metric
on a family of stable bundles. Since $H^1(S,\mathcal O_S)=0$,
$\det\mathcal E$ is the product of a fixed line bundle on $S$ and a
line bundle on $C$; on $Z$ its determinant differs from the pullback
only by an exceptional divisor trivial over $C^*$. Hence the determinant
metric can be normalized over $C^*$ to a product of a harmonic metric
over $S$ and a metric from $C^*$, a conformal change leaving
$\operatorname{End}(\mathcal E,H_0)$, and hence $R$, unchanged. The
determinant-normalized Hermitian--Einstein metrics $H_{\mathrm{HE}}$ are
unique and vary smoothly over $C^*$, and their Petersson--Weil form is
$$
 \Omega=\frac12p_{C*}
       \left(\Delta(\mathcal E,H_{\mathrm{HE}})\wedge p_S^*\omega_S\right)
 =\frac{r}{4\pi^2}\omega_{\mathrm{WP}},
$$
the last equality by \cite[Proposition~4.3, equation~(4.6)]{BiswasSchumacherWP}.
It is positive and strictly positive wherever the classifying map has
nonzero derivative. The secondary Bott--Chern formula and the
Hermitian--Einstein heat flow give a real Donaldson functional
$\mathcal D$ on $C^*$, normalized by its value at the initial metric,
with
\begin{equation}\label{eq:curvewise-donaldson-potential}
 \Omega=R+\sqrt{-1}\partial\bar\partial\mathcal D,
 \qquad \mathcal D\leq0;
\end{equation}
this is the heat-flow and potential identity of
\cite[\S5.2]{BiswasSchumacherWP}, the functional carrying the fixed
positive constant by which $\Omega$ has been normalized so that its
monotonicity is unchanged. Equivalently one may use the trace-free
curvature and the projective Hermitian--Einstein heat flow throughout.

On a disk where $R=\sqrt{-1}\partial\bar\partial\rho$ the function
$\rho+\mathcal D$ is subharmonic on the punctured disk and bounded above,
and so extends subharmonically across the puncture. As $\mathcal D$ is a
single global scalar function, these extensions define an $L^1$ extension
$\widetilde{\mathcal D}$ on $C$, and
$$
 T=R+\sqrt{-1}\partial\bar\partial\widetilde{\mathcal D}
$$
is a global positive current of class $[R]$, no boundary point mass
being chosen independently. Equation
\eqref{eq:curvewise-initial-current-class} proves the nonnegativity.
If the classifying map is nonconstant, $\Omega$ is strictly positive
on a nonempty open subset of $C^*$, so $\int_C T>0$.
\end{proof}

\begin{lemma}\label{lem:reflexive-hull-curve-family}
Let $\mathcal F$ be a flat family of slope-stable torsion-free coherent
sheaves on $S$, parametrized by a compact smooth curve $C$. Put
$\mathcal E=\mathcal F^{**}$ on the threefold $S\times C$ and
$\mathcal Q=\mathcal E/\mathcal F$. Then $\mathcal E$ is flat over $C$,
its fibers are slope-stable torsion-free coherent sheaves, and its
general fiber is locally free. Moreover, $\mathcal Q$ is flat over $C$
with zero-dimensional fibers of a constant length $\ell$.
\end{lemma}

\begin{proof}
The sheaf $\mathcal F$ is torsion-free on $S\times C$. The reflexive
sheaf $\mathcal E$ is locally free outside a finite set and is flat over
$C$; at an exceptional point it has depth at least two, so quotienting
by a local base parameter leaves depth at least one and $\mathcal E_t$
is torsion-free for every $t\in C$. In the base-change exact sequence
$$
 0\longrightarrow\operatorname{Tor}^{\mathcal O_C}_1
                  (\mathcal Q,\mathcal O_t)
 \longrightarrow\mathcal F_t\longrightarrow\mathcal E_t
 \longrightarrow\mathcal Q_t\longrightarrow0
$$
the Tor sheaf is supported in dimension at most one on $S$ and hence
vanishes, because $\mathcal F_t$ is torsion-free; this proves the
flatness of $\mathcal Q$ and the injectivity of
$\mathcal F_t\longrightarrow\mathcal E_t$. Since
$\det\mathcal F_t=\det\mathcal E_t$, we have $c_1(\mathcal Q_t)=0$, and
no one-dimensional part of $\mathcal Q_t$ can be an effective divisor of
positive degree against $\omega$; thus $\mathcal Q_t$ is
zero-dimensional, of constant length $\ell$ by flatness.
A destabilizing subsheaf of $\mathcal E_t$ would intersect
$\mathcal F_t$ in a subsheaf with the same rank and first Chern class,
contradicting stability. Finally, the singularities of $\mathcal E$
lie over finitely many parameters, so its general fiber is locally free.
\end{proof}

\begin{proposition}\label{prop:moduli-curve-positivity}
For the classes in \eqref{eq:moduli-donaldson-classes}, every irreducible
compact curve $D\subset M$ satisfies
$$
 A\cdot D\geq0.
$$
If $A\cdot D=0$, then $B\cdot D>0$.
\end{proposition}

\begin{proof}
Let $f:C\longrightarrow D\subset M$ be the normalization. Since
$H^2(C,\mathcal O_C^*)=0$, the stable-sheaf gerbe pulled back to $C$ is
trivial and an actual family $\mathcal F$ on $S\times C$ inducing $f$
exists. Write $c_1(\mathcal F)=p_S^*\xi+p_C^*\eta$ with
$\eta\in H^2(C,\mathbb Z)$. By the Mukai formula,
$$
\begin{aligned}
 f^*A
 &= -r p_{C*}\bigl(p_S^*\omega\wedge\operatorname{ch}_2(\mathcal F)\bigr)
       +(\omega\cdot\xi)\eta\\
 &= \frac12p_{C*}\bigl(\Delta(\mathcal F)\wedge p_S^*\omega\bigr).
\end{aligned}
$$
Apply Lemma~\ref{lem:reflexive-hull-curve-family} and retain its
notation. In $K$-theory $\mathcal F=\mathcal E-\mathcal Q$ and
$c_1(\mathcal F)=c_1(\mathcal E)$, so
\begin{equation}\label{eq:moduli-curve-discriminant-splitting}
 A\cdot D
 =\frac12\int_{S\times C}\Delta(\mathcal E)\wedge p_S^*\omega
   +r\int_{S\times C}\operatorname{ch}_2(\mathcal Q)\wedge p_S^*\omega,
\end{equation}
both terms being nonnegative: the first by
Lemma~\ref{lem:curvewise-donaldson-degree}, the second because
$\operatorname{ch}_2(\mathcal Q)$ is the effective one-cycle of the
finite flat sheaf $\mathcal Q$.

Suppose that $A\cdot D=0$. Then the first term of
\eqref{eq:moduli-curve-discriminant-splitting} vanishes and the
classifying map of $\mathcal E$ is constant, with value a stable vector
bundle $E_0$ since the general fiber is locally free; by simplicity and
cohomology and base change
$$
 \mathcal E\simeq p_S^*E_0\otimes p_C^*L
$$
for a line bundle $L$ on $C$. The second term vanishes precisely when
every component of $\operatorname{Supp}\mathcal Q$ maps to a point of
$S$, so the support cycle of $\mathcal Q_t$ is constant.

Twisting by $p_C^*L^{-1}$, which leaves the Mukai classes unchanged,
turns $\mathcal F$ into a family of quotients of the fixed bundle $E_0$
with a fixed support cycle $z$, so $f$ factors through the punctual Quot
space $Q_z(E_0)$. Let $\mathcal Q$ again denote its universal quotient,
let $p:S\times Q_z(E_0)\to Q_z(E_0)$ be the projection, and put
$V=p_*\mathcal Q$. Grothendieck--Riemann--Roch, using $c_1(S)=0$, gives
$$
 A|_{Q_z(E_0)}=0,\qquad
 B|_{Q_z(E_0)}=r c_1(\det V),
$$
$A$ and $B$ being pulled back by the natural map $Q_z(E_0)\to M$:
only the rank component $-r$ of $(-r,0,a)$ contributes to the
degree-two pushforward of $-\operatorname{ch}(\mathcal Q)$, and
$c_1(V)=p_*\operatorname{ch}_3(\mathcal Q)$.

All quotients with support cycle $z$ are annihilated by a fixed
sufficiently high power of the ideal of $\operatorname{Supp}(z)$, hence
are quotients of one fixed finite-dimensional vector space; the
equations expressing compatibility with the local algebra cut out
$Q_z(E_0)$ in a Grassmannian, and $\det V$ is the Pl\"ucker line bundle,
which is ample. Since $f$ is nonconstant, $B\cdot D>0$.
\end{proof}

Working on the parameter curve avoids the need to deform a
given rational curve along the Hodge locus of its class. The following
criterion uses precisely the two inequalities just proved.

\section{A criterion for the K\"ahlerness of a bimeromorphic model of
a compact hyperk\"ahler manifold}
\label{Kahler criterion section}

For a compact complex manifold $X$ bimeromorphic to a compact
hyperk\"ahler manifold, denote by $\mathcal C_X^+$ the component of
$\{\alpha\in H^{1,1}(X,\mathbb R):q(\alpha)>0\}$ contained in
the pseudo-effective cone. This is the positive cone of
\cite[Section~3.2]{Per}; under the identification of second cohomology
by a bimeromorphism it corresponds to the positive cone of a
compact hyperk\"ahler model \cite[Theorem~4.3(i)]{Bou04}.

We first record the part of the birational comparison argument
which only requires nonnegative intersections.

\begin{lemma}\label{nonnegative birational comparison}
Let $f:X\dashrightarrow Y$ be a bimeromorphism between compact
complex manifolds with trivial canonical bundles. Let
$a\in H^2(X,\mathbb R)$ and $a_Y=f_*a\in H^2(Y,\mathbb R)$
be nonnegative on all rational curves in their respective
manifolds, and assume that $a_Y$ is the class corresponding to $a$
under the degree-two cohomology identification induced by $f$. On a
common resolution
$$
X\xleftarrow{p}W\xrightarrow{q}Y
$$
one has $p^*a=q^*a_Y$.

Suppose, in addition, that $b\in H^2(X,\mathbb R)$ is strictly
positive on every rational curve $C\subset X$ with $a\cdot C=0$,
and that $a_Y+\varepsilon f_*b$ is a K\"ahler class on $Y$ for
some $\varepsilon>0$. Then $f$ is a biholomorphism and
$a+\varepsilon b$ is a K\"ahler class on $X$.
\end{lemma}
\begin{proof}
The map $f$ is an isomorphism in codimension one, and the
exceptional prime divisors $E_i$ of $p$ and $q$ coincide. Write
$$
p^*a=q^*a_Y+\sum_i a_i[E_i].
$$
Choose $n_i>0$ such that $-\sum_i n_iE_i$ is $p$-ample. Suppose
$a_1<0$ with $-a_1/n_1$ maximal among the positive numbers $-a_i/n_i$.
A general rational curve $R\subset E_1$ contracted by $p$ satisfies
$E_i\cdot R\ge0$ for $i\ne1$ and $\sum_i n_iE_i\cdot R<0$, so the
computation of \cite[Proposition~2.1]{Huybrechts03} gives
$$
-\sum_i a_iE_i\cdot R
\le -\frac{a_1}{n_1}\sum_i n_iE_i\cdot R<0,
$$
contradicting $0=p^*a\cdot R=q^*a_Y\cdot R+\sum_i a_iE_i\cdot R$ with
$q^*a_Y\cdot R\ge0$. Hence $a_i\ge0$ for every $i$. Each sign
inequality is taken on a resolution obtained by blowups over the
corresponding side, so the two may be compared on a common dominating
resolution without assuming a factorization by blowups over both sides.
Repeating the argument on a resolution projective over $Y$
gives the reverse inequality, and pulling both inequalities to a
common resolution proves $p^*a=q^*a_Y$. Only nonnegative intersections
were used.

Put $\alpha=a+\varepsilon b$ and $\alpha_Y=a_Y+\varepsilon f_*b$. If a
rational curve $R$ is contracted by $q$ and $p(R)$ is a curve, the
first assertion gives $a\cdot p(R)=0$, hence $\alpha\cdot p(R)>0$;
for rational curves contracted by $p$ the corresponding intersections
on $Y$ are nonnegative because $\alpha_Y$ is K\"ahler. Applying the
preceding coefficient argument to $p^*\alpha-q^*\alpha_Y$ gives
$p^*\alpha=q^*\alpha_Y$. A rational curve contracted by $p$ cannot map
to a curve under $q$, since the last equality would give zero
intersection with the K\"ahler class $\alpha_Y$. On a resolution
projective over $X$ the fibers of $p$ are rationally chain connected,
so $q$ is constant on them and $f$ descends to a holomorphic map. Its
Jacobian is a nonzero holomorphic section of the trivial line bundle
$K_X\otimes f^*K_Y^{-1}$, hence has no zeros; as $f$ has degree one it
is a biholomorphism, and $\alpha=f^*\alpha_Y$ is K\"ahler.
\end{proof}

\begin{proposition}\label{two classes Kahler criterion}
Let $X$ be a compact holomorphically symplectic manifold
bimeromorphic to a compact hyperk\"ahler manifold of
$K3^{[n]}$-type, and let $q$ be the Beauville--Bogomolov form on
$H^2(X,\mathbb Z)$ transported from the model.
Let $A,B\in H^{1,1}(X,\mathbb R)$ satisfy
the following conditions:
\begin{enumerate}
\item $A$ belongs to $\mathcal C_X^+$;
\item $A\cdot C\ge0$ for every rational curve $C\subset X$;
\item $B\cdot C>0$ for every rational curve $C\subset X$
      with $A\cdot C=0$;
\item no nonzero class in $\operatorname{NS}(X)$ is
      orthogonal to both $A$ and $B$.
\end{enumerate}
Then $A+\varepsilon B$ is a K\"ahler class for every
sufficiently small $\varepsilon>0$.
\end{proposition}
\begin{proof}
Fix a bimeromorphic compact hyperk\"ahler model $Y_0$ and identify
second cohomology by the induced isometry, so prime divisors on $X$
correspond to prime divisors on $Y_0$. Let $D$ be a prime divisor of
negative square.
Its transform on $Y_0$ is uniruled by \cite[Proposition~5.6]{Markman11}.
On a resolution of $D$ the restricted symplectic form has generic rank
$2n-2$ and descends through the rational quotient, whose general fiber is
therefore a rational curve; fiber integration gives a rational curve
$C_D\subset X$ and $c_D>0$ with
$$
q(L,[D])=c_D\,L\cdot C_D
\qquad\text{for every }L\in H^{1,1}(X,\mathbb R)
$$
(closed representatives exist since $X$ belongs to the Fujiki class
$\mathcal C$). Hence
$$
q(A,[D])\ge0,\qquad
q(A,[D])=0\ \Longrightarrow\ q(B,[D])>0,
$$
while for a prime divisor of nonnegative square the positive-cone
inequality gives $q(A,[D])>0$.

The image of $A$ therefore lies in the closure of the birational
K\"ahler cone of $Y_0$. The hyperplanes defined by primitive monodromy
birationally minimal (MBM) classes are locally finite in the
positive cone: their squares are bounded by \cite[Theorem~5.3]{AVcone},
and a compact subset of the positive cone can meet the orthogonals of
only finitely many lattice vectors with these bounded negative squares.
For $K3^{[n]}$-type with $n\ge2$ one has $b_2=23$, as required by that
theorem, and for $n=1$ the same assertion follows from the $(-2)$-classes
of a K3 surface. Thus $A+\varepsilon B$ lies in the fundamental
exceptional chamber for all sufficiently small $\varepsilon>0$, and by
hypothesis (4) a hyperplane through $A$ cannot contain the whole line
$A+\mathbb RB$, so after decreasing $\varepsilon$ this ray lies in one
component of the complement of the MBM hyperplanes. Hypothesis (4) also
lets us choose a point on the ray orthogonal to no nonzero integral
$(1,1)$-class, each such class excluding at most one value of
$\varepsilon$. By \cite[Theorem~5.4(2) and Proposition~5.6]{Markman11}
that point is the pullback of a K\"ahler class on a bimeromorphic
compact hyperk\"ahler model $Y$. Since K\"ahler chambers are the
components cut out by the MBM hyperplanes, the whole ray segment lies in
the same chamber. Thus for the induced map
$f:X\dashrightarrow Y$ the class
$A_Y=f_*A$ is nef and $A_Y+\varepsilon f_*B$ is K\"ahler for all
sufficiently small $\varepsilon>0$. Lemma
\ref{nonnegative birational comparison} concludes.
\end{proof}

Combining Theorem~\ref{ddbar HK limit}
with the preceding divisor argument gives the following
K\"ahlerness criterion; compare \cite[Theorem~1.19]{Per}.
No restriction on the deformation type of the compact hyperk\"ahler
manifold is needed.

\begin{corollary}\label{repaired Perego Kahler criterion}
Let $X$ satisfy the hypotheses of
Theorem~\ref{ddbar HK limit}, and let
$\alpha\in H^{1,1}(X,\mathbb R)$ belong to $\mathcal C_X^+$.
Suppose that $\alpha\cdot C>0$ for every rational curve $C\subset X$
and that
$$
q(\alpha,\beta)\ne0
\qquad\text{for every nonzero }
\beta\in H^{1,1}(X)\cap H^2(X,\mathbb Z).
$$
Then $\alpha$ is a K\"ahler class and $X$ is a compact hyperk\"ahler
manifold.
\end{corollary}
\begin{proof}
Theorem~\ref{ddbar HK limit} gives a
bimeromorphic compact hyperk\"ahler model $Y_0$. The divisor calculation
in the proof of Proposition~\ref{two classes Kahler criterion}
uses only holomorphic symplecticity and shows that
$q(\alpha,[D])>0$ for every prime divisor $D$, so the image of
$\alpha$ lies in the fundamental exceptional chamber of $Y_0$. Its
assumed genericity and \cite[Theorem~5.4(2) and
Proposition~5.6]{Markman11} give a bimeromorphic compact hyperk\"ahler
model
$Y$ on which the corresponding class is K\"ahler, and the strict
comparison of rational curves in \cite[Proposition~2.1]{Huybrechts03}
makes the bimeromorphism $X\dashrightarrow Y$ an isomorphism. Hence
$\alpha$ is K\"ahler.
\end{proof}

\section{Proof of Theorem~\ref{main moduli Kahler}: K\"ahlerness of the
moduli space}
\label{moduli proof section}
\begin{proof}
The rank-one case is a Hilbert scheme of points on $S$.
The case $v^2=0$ is a K3 surface by \cite{PT}. We may
therefore assume $r\ge2$ and $v^2>0$.
Replacing $\omega$ within its $v$-chamber does not change $M$.
We choose it outside the countably many hyperplanes
$(rD-s\xi)^\perp$ with $D\in\operatorname{NS}(S)$,
$s\in\mathbb Z$, and $rD-s\xi\ne0$.

By \cite[Theorem~1.1]{PT}, the manifold $M$ is compact,
holomorphically symplectic and deformation
equivalent to a Hilbert scheme of points on a K3 surface.
Theorem~\ref{moduli ddbar from diagonal} gives the
$\partial\bar\partial$-lemma, hence $M$ is strongly Gauduchon by
Lemma~\ref{closed Gauduchon completion}. The relative moduli family
of \cite[Lemma~4.1]{Per} exhibits $M$ as a limit of compact
hyperk\"ahler manifolds. Hence
Theorem~\ref{ddbar HK limit} makes $M$
bimeromorphic to a compact hyperk\"ahler manifold of
$K3^{[n]}$-type. Consequently $M$ belongs to the Fujiki class
$\mathcal C$, the Beauville--Bogomolov form $q$ and the positive cone
$\mathcal C_M^+$ are defined on $M$, and the
$\partial\bar\partial$-decomposition of $H^2(M,\mathbb C)$ agrees with
its Hodge decomposition.

The Mukai homomorphism
$\lambda_v:v^\perp\to H^2(M,\mathbb Z)$ is an integral isometry
by \cite[Proposition~5.2]{PT}. It is now also a Hodge isometry:
the universal Chern character has pure bidegree in Bott--Chern
cohomology, so the defining correspondence respects Hodge type;
the $\partial\bar\partial$-lemma identifies these decompositions
with de Rham cohomology. Set
$$
A=\lambda_v(0,r\omega,\omega\cdot\xi),
\qquad
B=\lambda_v(-r,0,a).
$$
Both Mukai vectors belong to $v^\perp\otimes\mathbb R$, and
$q(A)=r^2\omega^2>0$. Proposition~\ref{prop:moduli-curve-positivity}
gives
$$
A\cdot C\ge0,
\qquad
A\cdot C=0\ \Longrightarrow\ B\cdot C>0
$$
for every compact curve $C\subset M$.

In the relative moduli family of \cite[Lemma~4.1]{Per},
projective K3 fibers approach $S$, and their moduli spaces $M_t$
are projective compact hyperk\"ahler manifolds by \cite[Theorem~3.4]{PT}.
Let $A_t$ be the Donaldson class on $M_t$ defined by
\eqref{eq:moduli-donaldson-classes} and let $h_t$ be an ample
integral class on $M_t$. Proposition~\ref{prop:moduli-curve-positivity},
applied to a complete-intersection curve, together with the polarized
Fujiki identity gives
$$
0\le\int_{M_t}A_t h_t^{2n-1}
=c_Mq(A_t,h_t)q(h_t)^{n-1},
$$
$c_M>0$ the Fujiki constant. Since $q(A_t)>0$ while
$h_t^\perp\cap H^{1,1}(M_t,\mathbb R)$ is negative definite, we get
$q(A_t,h_t)>0$, so $A_t$ lies in the positive cone and is
pseudo-effective \cite[Theorem~4.3(i)]{Bou04}.

By Lemma~\ref{closed Gauduchon completion}, choose a real closed
$(4n-2)$-form $\Psi$ on $M$ with strictly positive
$(2n-1,2n-1)$ component. For positive currents $T_t$ with $[T_t]=A_t$
on the approaching fibers and smoothly varying Hermitian forms
$\gamma_t$, the same estimate gives
$$
\int_{M_t}T_t\wedge\gamma_t^{2n-1}
\le C\langle A_t,[\Psi]\rangle\le C';
$$
a subsequence converges to a closed positive current $T$ with $[T]=A$.
Since $q(A)>0$, it follows that $A\in\mathcal C_M^+$.

Finally, if $w=(s,D,b)\in(v^\perp)^{1,1}$ is integral, then
$$
q(A,\lambda_v(w))=\omega\cdot(rD-s\xi),
\qquad
q(B,\lambda_v(w))=rb-sa.
$$
By the choice of $\omega$, vanishing of the first pairing forces
$rD=s\xi$; if the second also vanishes then $rb=sa$, and together
with $w\in v^\perp$ these identities give $s v^2=0$, hence
$s=D=b=0$. The hypotheses of Proposition~\ref{two classes Kahler criterion}
are satisfied, so $M$ is K\"ahler. Its simple connectedness
and the one-dimensionality of its space of holomorphic two-forms
follow from \cite{PT} and the $\partial\bar\partial$-lemma.
Therefore $M$ is a compact hyperk\"ahler manifold.
\end{proof}

\end{document}